\documentclass[12pt]{article}
\usepackage[english]{babel}
\usepackage{amsfonts, dsfont, amsmath,amssymb,amsthm,mathrsfs, amsbsy}
\usepackage{fullpage}
\usepackage{graphicx,graphics}
\usepackage[dvipsnames]{xcolor}
\usepackage{epsfig}
\usepackage{latexsym}
\usepackage[applemac,utf8]{inputenc}
\usepackage[colorlinks=true]{hyperref}
\usepackage{enumitem}
\usepackage[toc,page]{appendix}
\usepackage{bbm}
\usepackage{stmaryrd}
\usepackage{comment}
\usepackage{pgf,tikz}
\usepackage{enumitem}
\usepackage{mathtools}
\usepackage{ulem} 
\usepackage{textcomp}
\usepackage[cal=euler]{mathalfa}
\usepackage[mono=false]{libertine}
\useosf
\usepackage[letterpaper,top=2cm,bottom=2cm,left=2cm,right=2cm,marginparwidth=1.5cm]{geometry}
\usepackage[font={small,sf}, labelfont={sf,bf}, margin=1cm]{caption}
\usepackage{csquotes}

\usepackage{savesym}
\savesymbol{addspace}
\usepackage{musixtex}
\restoresymbol{MS}{addspace}

\usepackage[%
backend=biber,
hyperref=true,
style=ext-alphabetic,
url=true,
isbn=false,
giveninits=true,
maxnames=3,
backref=false,
date=year,
abbreviate=true,
eprint=true,
articlein=false,
maintitleaftertitle=true,
language=english
]{biblatex}
\usepackage[auto]{contour}
\usetikzlibrary{shapes, angles, decorations.text, patterns, fadings,decorations.pathreplacing,fit}

\renewcommand{\geq}{\geqslant}
\renewcommand{\leq}{\leqslant}

\newtheorem{theorem}{Theorem}
\newtheorem*{theorem*}{Theorem}

\newtheorem{proposition}{Proposition}[section]
\newtheorem{lemma}{Lemma}

\theoremstyle{definition}

\theoremstyle{remark}

\theoremstyle{definition}

\date{}

\title{ 
\bf 
\begin{music}\trebleclef\end{music} \hspace{1em}
    \begin{minipage}{.6\linewidth}
        \centering
        Sweet Trims are made of Threes\\
        \normalsize \bf A c\`adl\`ag erasure of the Brownian tree
    \end{minipage}
\hspace{1em}\raisebox{.3em}{♪}\hspace{.2em}\raisebox{-.2em}{♪}\hspace{.2em}\raisebox{.2em}{♪} }

\author{Alessandra Caraceni\thanks{Scuola Normale Superiore di Pisa, \url{alessandra.caraceni@sns.it}}, \hspace{0.1cm} Nicolas Curien\thanks{Universit\'e Paris-Saclay, \href{mailto:nicolas.curien@universite-paris-saclay.fr;william.fleurat@universite-paris-saclay.fr;adrianus.twigt@universite-paris-saclay.fr}{\tt firstname.lastname@universite-paris-saclay.fr} }, \hspace{0.1cm} William Fleurat\footnotemark[2] \hspace{0.3cm}\&\hspace{0.1cm} Adrianus Twigt\footnotemark[2] }

\begin{document}
\maketitle

\vspace{-1em}
\begin{abstract} 

We present a simple trimming algorithm that generates nested uniform binary plane trees by removing leaves one-by-one using a best-of-three-match procedure. While its one-step transition specializes to the \L uczak--Winkler \& Caraceni--Stauffer coupling, its scaling limit provides a suprising  c\`adl\`ag erasure of Brownian trees, reminiscent of SLE theory.
\end{abstract}

\begin{figure}[h!]
    \begin{center}
        \includegraphics[width=1\linewidth]{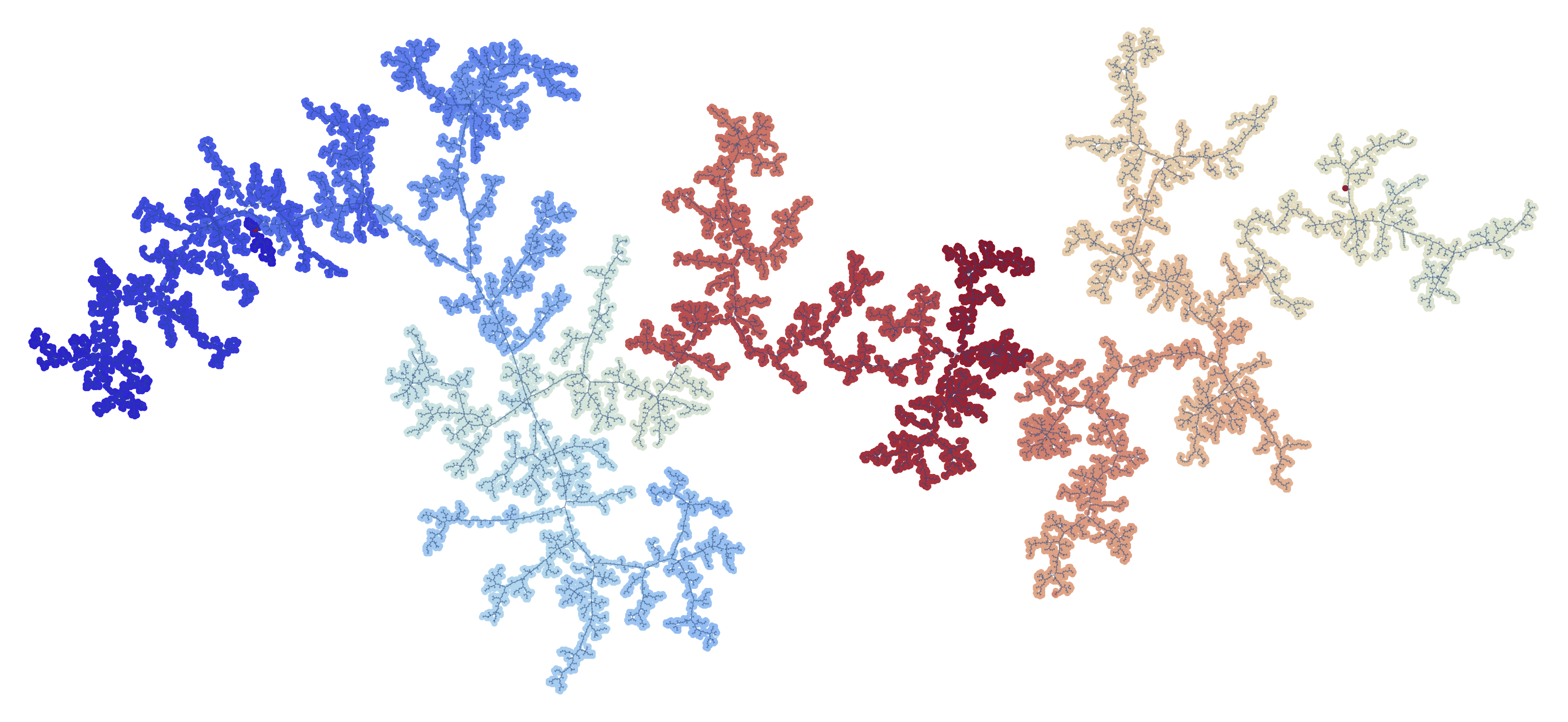}
    \end{center}
    \caption{A Brownian Continuum Random tree, with its vertices colored by time of erasure, from blue to red. See \href{https://www.imo.universite-paris-saclay.fr/~nicolas.curien/videos/erasure.mp4}{the corresponding video.}}
\end{figure}

\clearpage

\section{Introduction}
Trimming (and growing) algorithms to generate uniform random binary trees $(T_n : n \geq 0)$ have recently been proved to be particularly useful to study random structures \cite{LPS,CS23,addario2014growing} and to prove conditioned theorems \cite{fleurat2025tauberian}. Perhaps the most famous is Rémy's algorithm \cite{remy1985procede} which is a key tool in the field. A natural question, first addressed by \L uczak \& Winkler \cite{LW} is to ask for a coupling which is \emph{monotone}, in the sense that $T_{n-1}$ is obtained from $T_n$ by pruning leaves. Such algorithms have been explicited and studied in \cite{caraceni2020polynomial,caraceni2024random}. We provide below a simple algorithm in the case of binary trees that not only reproves and explains the {\L}uczak--Winkler \& Caraceni--Stauffer one-step construction, but also passes to a continuous coupling in the scaling limit: a continuous proces of Brownian Continuum Random Trees (CRT)~\cite{aldous1991continuum} that grow through a single point --- reminiscent of SLE-type growth. The algorithm is so simple that we can start its complete description right away.

In these pages, by \textbf{\textit{binary tree}} we mean a plane tree  $t$ whose vertices have either degree one (the \textit{leaves}) or degree three (the \textit{branching nodes}) and rooted at a leaf. 
Interpreting the root as the ancestor of a  genealogical tree, we shall use freely family  notions such as parent, descendant, sibling, \textit{etc}. The \textbf{size} $|t|$ of $t$ is its number of branching nodes, equivalently its number of leaves plus two.
A \textbf{\textit{labeled binary tree}} $\boldsymbol{t}=(t,\ell)$ is a binary tree $t$ together with a bijective labeling $v\mapsto\ell(v)$ of its leaves by $\{0,1,\dots,|t|+1\}$ such that the root has label $0$.
Given a binary tree $t$ and $v$ one of its nodes, the \textbf{\textit{fringe set}} $t(v)$ is the set of strict descendants of $v$ in $t$, that is the set of vertices which are disconnected from the root in $t\setminus\{v\}$. We then define $\mathrm{Cut}(t;v)=t\setminus t(v)$, which is again a binary tree, in which $v$ has been turned into a leaf. This cutting operation extends to labeled trees $\boldsymbol t=(t,v)$ by relabeling as follows: first, the new leaf $v$ gets assigned the minimal label in $t(v)$, and the leaves are renumbered increasingly by $\{0,1,\dots,|\mathrm{Cut}(t;v)|+1\}$.

Given a labeled binary tree $\boldsymbol t=(t,\ell)$, its \textbf{\textit{best-of-three erasure}} is the tree $\mathrm{Cut}(\boldsymbol t,v_*)$ where the node $v_*$ is chosen as follows.
Start at the branching node attached to the root leaf, and iteratively, suppose we are at a given ``active'' branching node $v$.
If $t(v)$ does not contain any branching node, then set  $v_*=v$.
Otherwise, among the three leaves with minimal labels in $t(v)$, at least two are descendants of the same child of $v$, itself a branching node, to which we move. See Figure \ref{fig:best-of-three-explained} for an illustration.

\begin{figure}[h!]
    \centering
    \includegraphics[width=0.8\linewidth]{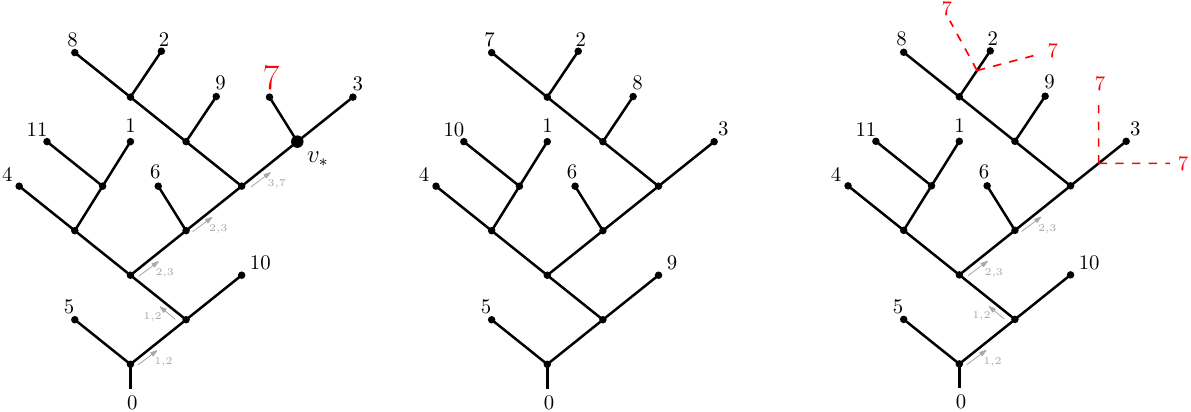}
    \caption{
        Left: On this example, the BoT leaf is labeled $7$. The gray arrows depict the moves to chose $v_*$. Center: the tree obtained by cutting at $v_*$ and relabeling. Right: The four different locations to re-insert a leaf labeled $7$.
    \label{fig:best-of-three-explained}}
    \label{fig:placeholder}
\end{figure}

By construction, $t(v_*)$ does not contain branch points: thus it is a pair of leaves. Hence $\mathrm{Cut}(\boldsymbol t,v_*)$ removes these two leaves, while $v_*$ inherits the minimal label between these two leaves. The other label is denoted by $\mathrm{BoT}(\boldsymbol t)$, encoding the \textbf{Best-of-Three leaf} (BoT leaf) that has been effectively erased. An important observation is that \textit{the sibling of the BoT leaf has a smaller label} than $\mathrm{BoT}(\boldsymbol t)$. Note also that the number of leaves (or the number of nodes) reduces by exactly one.  


\begin{proposition}[{A bijective explanation of Remark 2.5 in \cite{caraceni2024random}}]\label{prop:sports_question} Let $n\geq 2$. If \,$\boldsymbol{T}_n=(T_n, \ell_n)$ denotes a uniformly random labeled binary tree of size $n$, then the tree $ \boldsymbol{T}_{n-1}=(T_{n-1}, \ell_{n-1})$ obtained by best-of-three erasure is a uniformly random labeled binary tree of size $n-1$.
\end{proposition}
\begin{proof} 
    We shall prove in fact that the mapping $ (t_n, \ell_n) \longmapsto (t_{n-1}, \ell_{n-1})$
is $(4n-2)$-to-$1$, which gives the result. Suppose that  the BoT erasure of $\boldsymbol t_n$ yields $\boldsymbol t_{n-1}$, and set $j=\mathrm{BoT}(\boldsymbol t_n)$. Then one gets $\boldsymbol t_n$ back by: (i) relabeling increasingly by $\{1,\dots,n+1\}\setminus\{j\}$, giving some $\boldsymbol t'_{n-1}$, and (ii) grafting a leaf with label $j$ to the left or right of some \textit{allowed} leaf of $\boldsymbol t'_{n-1}$, as depicted in Figure \ref{fig:best-of-three-explained}. We claim that the number of allowed grafting locations is~$0$ when $j\in\{0,1\}$, it is $2$ when $j=2$, and lastly it is $4$ when $j\in\{3,\dots,n+1\}$ --- giving as claimed that $\boldsymbol t_{n-1}$ has $0+2+4(n-1)=4n-2$ preimages.

Recall that the BoT leaf in $\boldsymbol t_n$ is sibling to a leaf with a smaller label. This readily rules out the case $j\in\{0,1\}$, and shows that when the BoT leaf has label $2$ it must be a sibling to $1$, giving two grafting locations to reconstruct --- the left and right of the leaf labeled $1$. In the remaining case $j\in\{3,\dots,n+1\}$, the moves to choose $v_*$ are as follows: either the descendants of the active vertex contain at least $3$ labels  in $\{1,\dots,j-1\}$, in which case the next chosen vertex is the same  if one follows the BoT construction in either $\boldsymbol t_{n}$ or $\boldsymbol t'_{n-1}$ (this is because the renumbering and grafting only concern leaves with label larger than $j-1$); or there are exactly two leaves $v,v'$ of $\boldsymbol t'_{n-1}$ with labels in $\{1,\dots,j-1\}$ above the active vertex. Since the BoT leaf of $\boldsymbol t_n$ has label $j$ and is sibling to a leaf with a smaller label, we deduce that $\boldsymbol t_{n}$ must have been obtained by grafting a leaf with label $j$ at the left or right of either $v$ or $v'$. This gives $4$ possible grafting locations. They all give $\boldsymbol t_{n-1}$ under BoT erasure: the grafting creates a pair of leaves with two of the three minimal labels above the active vertex, forcing to cut at their parent, so that the BoT leaf has the larger label between the two, namely $j$ as needed.
\end{proof}

By coherence, we get a chain $( \boldsymbol{T}_k : k \geq 0)$ of uniformly random labeled binary trees with of size $k$ such that $\boldsymbol{T}_{k-1}$ is the BoT erasure of $\boldsymbol{T}_k$ for all $k \geq 1$. Our main result is a scaling limit for this chain.
See~\cite{le2005random} for background on the Gromov--Hausdorff metric $d_{\mathrm{GH}}$ and the Brownian CRT $\mathcal T$ --- that we normalize so that $n^{-1/2}\cdot T_n\to \mathcal T$, that is $\mathcal T$ is $\sqrt 2\cdot\mathcal T_{\mathrm{Aldous}}$ defined in~\cite{aldous1991continuum}. We consider the metric $D(X,Y)=\sup_t d_{\mathrm{GH}}(X_t,Y_t)$ to state a convergence of a \textit{process} of (isometry classes of) metric spaces.

\begin{theorem}[A c\`adl\`ag erasure of the Brownian CRT] \label{thm:main} Let $ (\mathcal{T},d)$ be a Brownian CRT. There is a.s.~a c\`adl\`ag space filling function $\xi:[0,1] \to \mathcal{T}$ such that for all $t \in[0,1]$, the range $\overline{\xi[0,t]}  =: \mathcal{T}_t\subset \mathcal{T}$ is a closed subtree of \,$\mathcal{T}$ having the same law as $\sqrt{t} \cdot \mathcal{T}$. Furthermore we have the
convergence in distribution, with respect to the metric $D$,
    $$ \left( \frac{1}{ \sqrt{n}} T_{[nt]} : t \in [0,1]\right) \xrightarrow[n \to \infty]{(d)} \left( \mathcal{T}_t\right)_{t \in [0,1]}.$$
\end{theorem}

\textbf{Motivation}: Our discrete construction specializes without pain to the coupling of binary trees of size $n+1$ and $n$ discovered by {\L}uczak \& Winkler \cite{LW} and explicited by Caraceni \& Stauffer \cite{caraceni2020polynomial}. More precisely, the choice of the leaf to erase  considered in \cite{caraceni2024random} is precisely the erasure of best-of-three leaf for an independent uniformly random labeling of the tree, see in particular Remark 2.5 ``A sports question" in \cite{caraceni2020polynomial}. However, in the above references the labeling is resampled at each step to produce the chain. This gives a chain of nested trees but where the erased leaf at step $n$ and $n-1$ are typically far from each other and whose scaling limits have been constructed in \cite{curien2025growingselfsimilarmarkovtrees} (in the much general framework of self-similar Markov trees \cite{bertoin2024self}) --- there, trees grow globally through countably many points. In a sense, the coupling of the present work can be thought of as the ``best possible'' growing process of labeled binary trees through their leaves. A similar phenomenon has recently been uncovered for Rémy's algorithm: if the leaf to erase is resampled at each step then the sequence of renormalized trees converge almost surely, while if they are coupled in the best possible way \cite{BBJ17}, this yields to a non-trivial diffusion \cite{CurienMarzouk26}. This paper is one part of a bigger research effort aimed at building from the discrete various tree-valued diffusion processes with Brownian CRT as invariant measures see \cite{CurienMarzouk26,curien2025growingselfsimilarmarkovtrees,FPRW23,evans2006rayleigh,LMW20}.

\medskip \noindent \textbf{Acknowledgments:}  The last three authors are supported by ``SuperGrandMa", the  ERC Consolidator Grant No 101087572. 

\section{Proof of Theorem \ref{thm:main}}

The proof of our main result is straightforward after an important observation in the discrete setting: the order of erasure is well-compatible with its approximation using only the $k$ smallest labels. 
To state it formally, let us introduce some terminology and notation. 

Fix  $\boldsymbol{t}_n=(t_n, \ell_n)$ a labeled binary tree of size $n$ and write $(\boldsymbol{t}_k : 0 \leq k \leq n)$ for the labeled binary trees obtained by successive BoT erasures. 
We let $(\mathrm{b}_{k},1\leq k \leq n)$ denote the branching nodes at which we successively cut --- in the notation of the preceding section, $\mathrm{b}_k$ is the branching node $v_*$ used in the BoT erasure of $\boldsymbol{t}_{n-k+1}$ to obtain $\boldsymbol{t}_{n-k}$. All in all, the sequence $(\mathrm{b}_{k},1\leq k \leq n)$ describes a total order $\prec_{\boldsymbol{t}_n}$ on the set $\mathbb{B}(t_n)$ of branching nodes of $t_n$,
$$ \mathrm{b}_1 \prec_{\boldsymbol{t}_n} \mathrm{b}_2 \prec_{\boldsymbol{t}_n} \cdots \prec_{\boldsymbol{t}_n} \mathrm{b}_n,$$
where the smallest branching node is the first erased, see Figure \ref{fig:compatibility} for an illustration.

We now consider \textbf{spanned subtrees}: for $2\leq \ell\leq n+1$, the $\ell$-span of $\boldsymbol t_n$, denoted by $\boldsymbol t_n^{\ell}$, is the union of the paths between the leaves with labels in $\{0,1,\dots,\ell\}$. This gives a \textbf{labeled unary-binary tree}, see Figure~\ref{fig:compatibility} --- these generalize labeled binary trees by allowing vertices of degree two. 
The definition of \textit{BoT erasure} readily extends to these trees  by treating the non-branching nodes (degree $2$) as fictitious. Again, repeated BoT erasure gives a total order $ \mathrm{b}_1^\ell \prec_{\boldsymbol{t}_n^{\ell}} \mathrm{b}_2^\ell \prec_{\boldsymbol{t}_n^{\ell}} \cdots \prec_{\boldsymbol{t}_n^{\ell}} \mathrm{b}_{\ell-1}^\ell$ on the set $\mathbb B(t_n^\ell)$ of branching nodes of $t_n^\ell$.
Note that the branching nodes $ \mathbb{B}(t_n^\ell)$ form a (generally disconnected) subset of $\mathbb{B}(t_n)$.
In fact, the order $\prec_{\boldsymbol{t}_n^{\ell}}$ refines the order $\prec_{\boldsymbol{t}_n}$, as seen in the following lemma.


\begin{figure}[h!]
    \centering
    \includegraphics[width=0.9\linewidth]{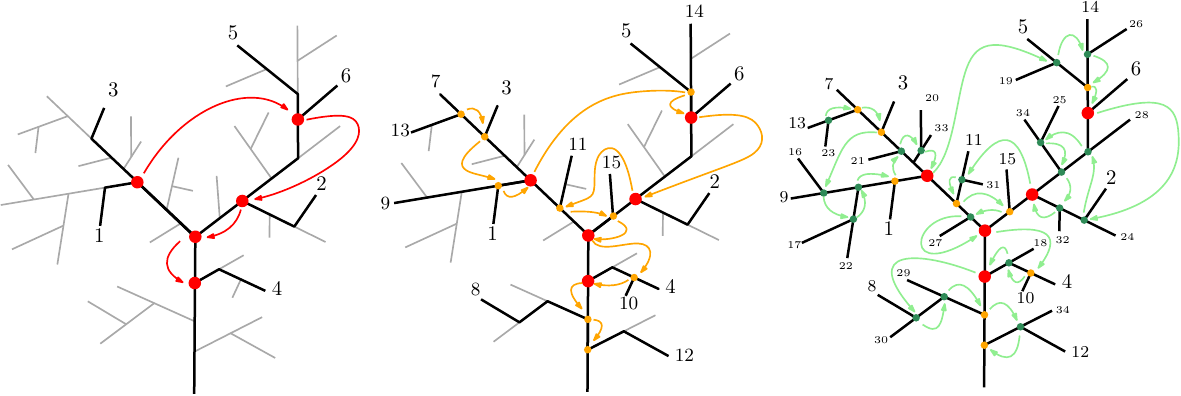}
    \caption{Illustration of the compatibility of order or erasure in spanned subtrees. Left: In the unary-binary tree (in bold) spanned by the root and the leaves $1,2,3,4,5,6$, the successive removal of the branch points is displayed in red. Middle and Right: Refining the order of erasure by revealing more leaves. \label{fig:compatibility}}
    
\end{figure}


\begin{lemma}[Compatibility] \label{lem:compatibility} For all  $2 \leq \ell \leq n$ and $ 1 \leq i \leq \ell-2$, we have:
    \begin{itemize}
    \item  $\mathrm{b}_i^\ell \prec_{\boldsymbol{t}_n} \mathrm{b}_{i+1}^\ell$ \textsc{(Compatibility with $\prec_{\boldsymbol{t}_n}$)},
    \item  any branch point $\mathrm{b} \in \mathbb{B}(t_n)$ lying in the interval $\mathrm{b}_i^\ell \prec_{\boldsymbol{t}_n} \mathrm{b} \prec_{\boldsymbol{t}_n} \mathrm{b}_{i+1}^\ell$ must belong to the subtree 
$$ \mathrm{b} \in t_n(\mathrm{b}_{i+1}^\ell) \backslash \left(\bigcup_{1 \leq j \leq i} t_n(\mathrm{b}_{j}^\ell) \right) \qquad \mbox{\textsc{(nesting)}}.$$
\end{itemize}
\end{lemma}

\begin{proof}

It suffices to show that the order is compatible when passing from $t_n^\ell$ to $t_n^{\ell + 1}$. Denote by $\mathrm{b} \in t_n^{\ell + 1}$ the \textit{branching} node that arises when revealing the leaf labeled $\ell + 1$ in $t_n^{\ell+1}$. As was already noted in the proof of Proposition~\ref{prop:sports_question}, the leaf labeled $\ell + 1$ only changes the outcome of a best of three comparison if it is among the majority of the $3$ leaves with smallest labels in the fringe set of an active vertex. The label $\ell+1$ being the largest in $t_n^{\ell + 1}$, this will not happen until
\begin{enumerate}[label=(\roman*)]
    \item the path to the next BoT leaf passes through the parent $\mathrm{b}_k^{\ell} \in t_n^\ell$ of $\mathrm{b}$ for some $1 \leq k \leq \ell$,
    \item the fringe set $t_n^{\ell + 1}(\mathrm{b}_k^\ell)$ is reduced to having only $1$ branching node $\mathrm{b}$ and $3$ leaves,
    \item the fringe set $t_n^{\ell + 1}(\mathrm{b})$ is reduced to having only $2$ leaves, among which the leaf labeled $\ell + 1$.
\end{enumerate}
Note that if the parent of $\mathrm{b}$ is $\mathrm{b}_1^\ell$, then $\mathrm{b}$ is erased first and if the parent of $\mathrm{b}$ is the root, then it is erased last. In any other case, the presence of $\ell + 1$ does not change the $\prec_{\boldsymbol{t}_n^{\ell + 1}}$-order of the $\mathrm{b}_j^\ell$ for $j \leq k - 1$. Only after the erasure of $\mathrm{b}_{k-1}^\ell$ does the above situation arise, in which case the leaf labeled $\ell + 1$ and thereby $\mathrm{b}$ are erased next. At this point we recover exactly the same tree as after $k - 1$ steps in the best-of-three erasure for $t_n^\ell$. Hence the order also remains unchanged for $j \geq k$, from which the first claim of the lemma follows. As for the second claim, let $1\leq i\leq \ell-2$ and let $\mathrm{b} \in \mathbb{B}(t_n)$ be in the interval $\mathrm{b}_i^\ell \prec_{\boldsymbol{t}_n} \mathrm{b} \prec_{\boldsymbol{t}_n} \mathrm{b}_{i+1}^\ell$.
On the one hand, since ancestors are erased after their descendants, $\mathrm{b} \notin \left(\bigcup_{1 \leq j \leq i} t_n(\mathrm{b}_{j}^\ell) \right)$.
On the other hand, when $\mathrm{b}_{i}^\ell$ has been cut while $\mathrm{b}_{i+1}^\ell$ has not yet been cut, the BoT comparisons lead to $\mathrm{b}_{i+1}^\ell$ in what remains of $\boldsymbol t_n^\ell$, but the same is true in $\boldsymbol t_n$ --- possibly continuing higher up after climbing to $\mathrm{b}_{i+1}^\ell$ --- since larger labels do not impact BoT comparisons beyond the three minimal labels above an active vertex. Hence, nodes which are cut (strictly) between $\mathrm{b}_{i}^\ell$ and $\mathrm{b}_{i+1}^\ell$ belong to $t_n(\mathrm{b}_{i+1}^\ell)$.  In particular, $\mathrm b$ does. 
%
%
%
\end{proof}

\paragraph{Construction of the limit process} Start with a Brownian CRT $\mathcal{T}$ of mass $1$, and conditionally on it, sample a countable collection $(X_i : i \geq 0)$ of iid  points in $\mathcal{T}$ distributed according to the uniform mass measure $\mu$. Those points are a.s.~all leaves and $X_0$ will be seen as the root of $\mathcal{T}$ while $X_i$ will be interpreted as the leaves labeled $1,2, ... $. This collection of points enables us to define an order on the branch points $\mathbb{B}(\mathcal{T})$ of $ \mathcal{T}$. Recall that a branch point $x \in \mathbb{B}(\mathcal{T})$ is a point such that $\mathcal{T} \backslash \{x\}$ has three connected components. It is well-known that there are a.s. countably many such points in $\mathcal{T}$ and that they are obtained as the union of all branch points yielding to the $(X_i : i \geq 0)$'s. More precisely, for each $\ell \geq 2$ there are exactly $\ell-1$ branch points belonging to the tree $\boldsymbol{\mathcal{T}}^\ell$ spanned by $\{X_0, X_1, ... , X_\ell\}$ and the union of those points over $\ell$ is precisely $\mathbb{B}(\mathcal{T})$. Extending the discrete definition, one can consider the BoT-removal order $\prec_{\boldsymbol{\mathcal{T}}^\ell}$ on $\mathbb{B}(\boldsymbol{\mathcal{T}}^\ell)$, and for any $ \mathrm{b} \in \mathbb{B}(\boldsymbol{\mathcal{T}}^\ell)$, we define its \textbf{reverse-time} as the  total $\mu$-measure of the subtrees erased before $\mathrm{b}$:

\begin{equation}
    \label{def:theta}
    \theta(\mathrm{b}) \quad := \quad \mu \left( \bigcup_{ \mathrm{b}' \preceq_{\boldsymbol{\mathcal{T}}^\ell} \mathrm{b}} \mathcal{T}( \mathrm{b}')\right),
\end{equation} where, as in the discrete setting, $\mathcal{T}(x)$ denotes the fringe set above the point $x$ in $\mathcal{T}$. Using  Lemma \ref{lem:compatibility} it is seen that the above definition is coherent in the sense that if $ \mathrm{b} \in \mathbb{B}(\mathcal{T}^{\ell}) \cap  \mathbb{B}(\mathcal{T}^{\ell'})$ then both definitions of $\theta( \mathrm{b})$ for $\ell,\ell'$ agree. Thus $\theta$ is defined on $\mathbb{B}(\mathcal{T})$ with values in $[0,1]$.

\begin{proposition}\label{prop:extension}A.s.~the inverse mapping $\theta^{-1} : [0,1] \to \mathbb{B}(\mathcal{T})$ can be uniquely extended into a c\`adl\`ag space filling function over $\mathcal{T}$ such that for any $t\in [0,1]$ the closed subset $\overline{\theta^{-1}[1-t,1]} := \mathcal{T}_t$ is a closed subtree of $\mathcal{T}$.
\end{proposition}
\begin{proof} For $\ell \geq 2$, let us denote by $C_1^{\ell}, ... ,C^{\ell}_{2\ell-1}$ the connected components of $\mathcal{T} \backslash \mathbb{B}( \boldsymbol{\mathcal{T}^\ell})$. By standard properties of the Brownian CRT and of its mass measure, the set $\mathbb{B}(\mathcal{T})$ is dense in the skeleton of $\mathcal{T}$. It follows by compactness and the fact that $\mu$ has no atoms that 
\begin{equation} \label{eq:diam}
\lim_{\ell \to \infty} \sup_{1 \leq i \leq 2 \ell-1} \mathrm{Diam}_{\mathcal{T}}( C_i^{\ell})  = 0 \quad \mbox{and} \quad \lim_{\ell \to \infty} \sup_{1 \leq i \leq 2 \ell-1} \mu( C_i^{\ell})  = 0. \end{equation} Let $ (\mathrm{b}_i^\ell : 1 \leq i \leq \ell-1)$ be the branch points of $\mathbb{B}( \mathcal{T}^\ell)$ ranked via $\prec_{ \boldsymbol{\mathcal{T}}^\ell}$. Then, the nesting property of Lemma~\ref{lem:compatibility} entails that  $\mathcal{T}( \mathrm{b}_i^\ell) \backslash \bigcup_{1 \leq j < i} \mathcal{T}( \mathrm{b}_j^\ell)$ is composed of two components among the $(C_k^\ell)$'s. Hence,  \begin{equation} \label{eq:petitarbre} \lim_{\ell \to \infty} \sup_{1 \leq i \leq \ell-1}  \mu\left(\mathcal{T}( \mathrm{b}_i^\ell) \backslash \bigcup_{1 \leq j < i} \mathcal{T}( \mathrm{b}_j^\ell)\right) = 0 \quad \mbox{and} \quad \lim_{\ell \to \infty} \sup_{1 \leq i \leq \ell-1}  \mathrm{Diam}_{\mathcal{T}}\left(\mathcal{T}( \mathrm{b}_i^\ell) \backslash \bigcup_{1 \leq j < i} \mathcal{T}( \mathrm{b}_j^\ell) \right) = 0.\end{equation} We deduce from the first point that the values taken by $\theta$ on $\mathbb B(\mathcal T^\ell)$ form a subdivision of $[0,1]$ whose maximal gap shrinks to $0$ as $\ell\to\infty$. The image of $\theta : \mathbb{B}(\mathcal{T}) \to [0,1]$ is thus dense in $[0,1]$; and to prove the existence of the c\`adl\`ag extension $\theta^{-1}$, it suffices to check that it admits a c\`adl\`ag modulus, that is:
\begin{equation} \label{eq:cadladmodulus}\inf_{0=t_0<t_1<... < t_n=1 } \ \   \max_{0 \leq i \leq n-1} \ \  \max_{u,v \in [t_i,t_{i+1})} {d}_\mathcal{T}\big( \theta^{-1}(u),\theta^{-1}(v)\big) =0. \end{equation}
Using Lemma \ref{lem:compatibility} and the second point of \eqref{eq:petitarbre}, this follows by taking $t_i = \theta( \mathrm{b}_i^\ell)$ and then $\ell \to \infty$.  That $ \mathcal{T}_t$ is a closed subtree for all $t \in [0,1]$ and that $\theta^{-1}$ is space filling both follow from the similar facts in the discrete passed to the scaling limit. \end{proof}

\begin{proof}[Proof of Theorem \ref{thm:main}] If $(X^{(n)}_i: 0 \leq i \leq n+1)$ are the leaves of $ \boldsymbol{T}_n$, and $\mu_n$ the uniform probability measure on the leaves of $\boldsymbol{T}_n$, then by \cite[Theorem 5]{CurienHaas2012} we have 
$$ \left( \frac{1}{ \sqrt{n}} T_n ; \left(X_i^{(n)} \right)_{0 \leq i \leq n},\mu_n \right) \xrightarrow[n \to \infty]{(d)} ( \mathcal{T} ; (X_i)_{i\geq 0},\mu),$$ for the $n$-pointed Gromov-Hausdorff--Prokhorov topology. By Skorokhod embedding, let us suppose that this convergence is almost sure and takes place for the $n$-pointed Hausdorff and Prokhorov distances in some underlying compact metric space $(E,d)$.  This, together with standard a.s.~properties of $ \mathcal{T}$ show that for any $\ell \geq 2$, the branching points  as well as the tree erased up to those points converge towards their respective continuous counterpart. More precisely, let us denote by $ \mathrm{b}_1^{\ell,n} \prec_{ \boldsymbol{T}_n^\ell} \cdots \prec_{ \boldsymbol{T}_n^\ell}  \mathrm{b}_{\ell-1}^{\ell,n}$ the branch points of $ \boldsymbol{T}_n^\ell$ ranked by BoT-erasure and write $\theta^{(n)}( \mathrm{b}_i^{\ell,n})$ their erasing order in $\boldsymbol{T}_n$, meaning that if $\theta^{(n)}( \mathrm{b}_i^{\ell,n}) = k \in \llbracket 1, n-1 \rrbracket$, then the branch point $\mathrm{b}_i^{\ell,n}$ is exactly the $k$-th branch point to be erased in $\boldsymbol{T_n}$ (not in $\boldsymbol{T}_n^\ell$). Then, for each $1 \leq i \leq \ell- 1$ it follows from Hausdorff convergence that $\mathrm{b}_i^{\ell,n} \to  \mathrm{b}_i^\ell$ in $E$, and from Prokhorov convergence that $\theta^{(n)}(\mathrm{b}_i^{\ell,n}) \to \theta(\mathrm{b}_i^{\ell})$, almost surely as $n \to \infty$. By Lemma~\ref{lem:compatibility}, we get for all $1 \leq i \leq \ell-1$,
$$ \frac{1}{\sqrt{n}}T_{n- \theta^{(n)}( \mathrm{b}_i^{\ell,n})} = \frac{1}{\sqrt{n}}\left(T_{n} \backslash  \bigcup_{1 \leq j \leq i} T_n( \mathrm{b}_j^{\ell,n})\right) \xrightarrow[n \to \infty]{a.s.}  \mathcal{T} \, \backslash  \bigcup_{1 \leq j \leq i} \mathcal{T}( \mathrm{b}_j^{\ell}) = \mathcal{T}_{1- \theta( \mathrm{b}_i^\ell)},$$
where the last equality comes from Proposition \ref{prop:extension}. This proves the convergence of the rescaled process in Theorem \ref{thm:main} with $\xi(t) = \theta^{-1}(1-t)$ along a dense set of times $t$. The uniform control in \eqref{eq:petitarbre} is sufficient to upgrade to uniform convergence of the process. 
\end{proof}

\printbibliography

\end{document}